\documentclass[11pt,twoside, leqno]{article}

\usepackage{amssymb}
\usepackage{amsmath}
\usepackage{mathrsfs}
\usepackage{amsthm}
\usepackage{amsfonts}
\usepackage{txfonts}
\usepackage{enumerate}
\usepackage{latexsym,amssymb}
\usepackage{color}

\allowdisplaybreaks

\def\ls{\lesssim}
\def\gs{\gtrsim}
\def\fz{\infty}

\def\r{\right}
\def\lf{\left}

\def\supp{{\mathop\mathrm{\,supp\,}}}

\def\rr{{\mathbb R}}

\def\rn{{{\rr}^n}}
\def\zz{{\mathbb Z}}
\def\nn{{\mathbb N}}
\def\cc{{\mathbb C}}

\newcommand{\wz}{\widetilde}

\newcommand{\cs}{{\mathcal S}}

\def\az{\alpha}
\def\lz{\lambda}
\def\blz{\Lambda}

\def\fai{\varphi}

\def\wz{\widetilde}

\def\ls{\lesssim}
\def\gs{\gtrsim}

\def\bbmo{{{\mathop\mathrm{BMO}}}}

\def\dsum{\displaystyle\sum}
\def\dint{\displaystyle\int}
\def\dfrac{\displaystyle\frac}
\def\dsup{\displaystyle\sup}
\def\dinf{\displaystyle\inf}

\newtheorem{theorem}{Theorem}[section]
\newtheorem{lemma}[theorem]{Lemma}

\newtheorem{proposition}[theorem]{Proposition}
\theoremstyle{definition}
\newtheorem{remark}[theorem]{Remark}
\newtheorem{definition}[theorem]{Definition}
\numberwithin{equation}{section}
\def\supp{{\mathop\mathrm{\,supp\,}}}

\def\lfz{\lfloor}
\def\rfz{\rfloor}
\numberwithin{equation}{section}

\begin{document}

\arraycolsep=1pt

\title{\bf\Large
Intrinsic Structures of Certain Musielak-Orlicz Hardy Spaces
\footnotetext {\hspace{-0.35cm}
2010 {\it Mathematics Subject Classification}. Primary: 42B30;
Secondary: 42B35, 46E30, 46B70.
\endgraf {\it Key words and phrases}.
Hardy space, Musielak-Orlicz function, Muckenhoupt weight, interpolation,
atom, Calder\'on-Zygmund decomposition.
\endgraf
This project is supported
by the National Natural Science Foundation of China (Grant Nos. 11501506,
11471042, 11571039 and 11671185).
Jun Cao is also partially supported by the Natural Science Foundation of Zhejiang University of Technology
(Grant No. 2014XZ011).}}
\author{Jun Cao, Liguang Liu\footnote{Corresponding author /
{\color{red}July 19, 2017.}}\,, Dachun Yang and Wen Yuan}
\date{}
\maketitle

\vspace{-0.6cm}

\begin{center}
\begin{minipage}{13.5cm}
{\small {\bf Abstract} \quad For any $p\in(0,\,1]$, let $H^{\Phi_p}(\mathbb{R}^n)$ be the
Musielak-Orlicz Hardy space associated with
the Musielak-Orlicz growth function $\Phi_p$, defined by setting, for any
$x\in\mathbb{R}^n$ and $t\in[0,\,\infty)$,
 $$
\Phi_{p}(x,\,t):=
\begin{cases}
\dfrac{t}{\log{(e+t)}+[t(1+|x|)^n]^{1-p}}& \qquad \textup{when}\ n(1/p-1)\notin
\mathbb N \cup \{0\};\vspace{0.2cm}\\
\dfrac{t}{\log(e+t)+[t(1+|x|)^n]^{1-p}[\log(e+|x|)]^p}& \qquad \textup{when}\ n(1/p-1)\in \mathbb N\cup\{0\},
\end{cases}
$$
which is the sharp target space of the bilinear decomposition of
the product of the Hardy space $H^p(\mathbb{R}^n)$
and its dual. Moreover, $H^{\Phi_1}(\mathbb{R}^n)$
is the prototype appearing in the real-variable theory of
general Musielak-Orlicz Hardy spaces. In this article, the authors find a new structure of
the space $H^{\Phi_p}(\mathbb{R}^n)$ by showing that, for any
$p\in(0,\,1]$,
$H^{\Phi_p}(\mathbb{R}^n)=H^{\phi_0}(\mathbb{R}^n)
+H_{W_p}^p(\rn)$ and, for any $p\in(0,\,1)$,
$H^{\Phi_p}(\mathbb{R}^n)=H^{1}(\mathbb{R}^n)
+H_{W_p}^p(\rn)$, where $H^1(\mathbb{R}^n)$ denotes the classical real Hardy space,
$H^{\phi_0}(\rn)$ the Orlicz-Hardy space associated with the Orlicz function
$\phi_0(t):=t/\log(e+t)$ for any $t\in [0,\infty)$
and $H_{W_p}^p(\mathbb{R}^n)$ the weighted Hardy space associated with
certain weight function $W_p(x)$ that is comparable to $\Phi_p(x,1)$ for any $x\in\mathbb{R}^n$.
As an application, the authors further establish an interpolation theorem of
quasilinear operators based on this new structure.}
\end{minipage}
\end{center}

\section{Introduction\label{s1}}

\hskip\parindent
The real-variable theory of the classical real Hardy spaces on the Euclidean space $\rn$ was initially developed by
Stein and Weiss \cite{SW60} and later by Fefferman and Stein \cite{FS72}.
For any $p\in(0,\infty)$, the classical real Hardy space $H^p(\rn)$ consists  of all
Schwartz distributions $f$ such that
$$f^+:= \sup_{t\in(0,\infty)} |\phi_t\ast f|\in L^p(\rn),$$
where $\phi$ is a function in the Schwartz class, $\int_{\rn}\phi(x)\,dx=1$
and $\phi_t(\cdot):=t^{-n}\phi(t^{-1}\cdot)$.
As one of the most important function spaces in harmonic analysis, $H^p(\rn)$ has many applications in various
fields of mathematics (see, for example,
\cite{FS72,St93,CLMS93,Lu95} and their references). Later, the theory of
$H^p(\rn)$ was extended to the setting
of the weighted Hardy spaces by Garc\'ia-Cuerva \cite{Ga79} and Str\"omberg,  Torchinsky \cite{ST89},
and also to the setting of the Orlicz-Hardy space by  Str\"omberg \cite{Str79} and Janson \cite{Ja80}.
Both of the latter two spaces
can be viewed as special cases of more general Musielak-Orlicz
Hardy spaces which were first introduced by Ky \cite{Ky14} (see also \cite{ylk17}
for a complete survey of the real-variable theory of Musielak-Orlicz
Hardy spaces).

The main aim of this article is to try to understand some intrinsic structure
of the Musielak-Orlicz Hardy space $H^{\Phi_p}(\rn)$
associated with the Musielak-Orlicz growth function
\begin{align}\label{Phip}
\Phi_{p}(x,\,t):=
\begin{cases}
\dfrac{t}{\log(e+t)+[t(1+|x|)^n]^{1-p}}& \ \text{when}\ n(1/p-1)\notin \nn\cup\{0\},\vspace{0.2cm}\\
\dfrac{t}{\log(e+t)+[t(1+|x|)^n]^{1-p}[\log(e+|x|)]^p}& \ \text{when}\ n(1/p-1)\in \nn\cup \{0\},
\end{cases}
\end{align}
where $x\in\rn$ and $t\in[0,\infty)$ (see \cite{BGK12,Ky14,BCKLYY}).
The precise definition of $H^{\Phi_p}(\rn)$ is as follows.
In what follows, we use $\cs(\rn)$ and $\cs'(\rn)$ to denote the space
of all Schwartz functions, equipped with the classical well-known topology,
and its dual space, equipped with the weak-$\ast$ topology.

\begin{definition}\label{defn-Hp}
Let $p\in(0,1]$ and  $\Phi_p$ be as in \eqref{Phip}.

\begin{enumerate}

\vspace{-0.2cm}

\item[\rm (i)]   For any $f\in\cs'(\rn)$ and $m\in\nn$,
the {\it non-tangential grand maximal function} $f_m^*$ of $f$ is defined by setting, for any $x\in\rn$
\begin{align}\label{grand-funct}
f_m^*(x):=\dsup_{{\fai}\in \mathcal{S}_m(\rn)}\,\dsup_{|y-x|<t,\, t\in (0,\fz)}\lf|f\ast {\fai}_t(y)\r|,
\end{align}
where $\fai_t(\cdot):=t^{-n}\fai(t^{-1}\cdot)$ for any $t\in(0,\infty)$, and
\begin{align*}
\mathcal{S}_m(\rn):=\lf\{\fai\in \mathcal{S}(\rn):\
\dsup_{|\az|\le m+1}\dsup_{x\in\rn} \lf(1+|x|\r)^{(m+2)(n+1)}|D^\az \fai (x)|\le 1\r\}.
\end{align*}

\item[\rm (ii)]
The {\it Musielak-Orlicz-Lebesgue space}
$L^{\Phi_p}(\rn)$ is defined to be the space of all measurable functions $f$ on $\rn$ such that
\begin{align*}
\lf\|f\r\|_{L^{\Phi_p}(\rn)}
:=\inf\lf\{\lz\in(0,\infty):\, \int_\rn \Phi_p(x, |f(x)|/\lz)\,dx\le 1\r\}<\infty.
\end{align*}
With $m$ being the largest integer not greater than $n(1/p-1)$, the {\it Musielak-Orlicz Hardy space}
$H^{\Phi_p}(\rn)$ is defined to be the space of all Schwartz distributions $f\in\cs'(\rn)$
such that
$$
\|f\|_{H^{\Phi_p}(\rn)}:= \|f^\ast_m\|_{L^{\Phi_p}(\rn)}<\infty,
$$
where $f_m^\ast$ is as in \eqref{grand-funct}.
\end{enumerate}
\end{definition}

Notice that the non-tangential grand maximal function $f^*$ in Definition \ref{defn-Hp} can
also be used to characterize the Hardy space
$H^p(\rn)$ when $p\in(0,1]$. Indeed, one has $\|f\|_{H^p(\rn)}\sim \|f^\ast\|_{L^p(\rn)}$
for any $p\in(0,1]$ and any $f\in\cs'(\rn)$ with the equivalent positive constants independent
of $f$.

One of main motivations for us to study the aforementioned
Musielak-Orlicz Hardy spaces $H^{\Phi_p}(\rn)$ for any $p\in(0,1]$
comes from the bilinear decomposition of the product of functions in Hardy space $H^p(\rn)$ and its dual space.
When $p=1$, Bonami et al. \cite{BGK12} established the
following sharp {bilinear} decomposition of the product of the Hardy space
$H^1(\rn)$ and its dual  space $\bbmo(\rn)$
\begin{align}\label{eq-H1}
H^1(\rn)\times\bbmo(\rn)\subset L^1(\rn)+H^{\rm log}(\rn),
\end{align}
where $H^{\rm log}(\rn)$ is just the Musielak-Orlicz Hardy
space $H^{\Phi_1}(\rn)$. The precise meaning of \eqref{eq-H1} is that
there exist two bounded bilinear operators
$S:\, H^1(\rn)\times \bbmo(\rn)\to L^1(\rn)$ and
$T:\, H^1(\rn)\times \bbmo(\rn)\to H^{\log}(\rn)$
such that the product, defined in the sense of $\cs'(\rn)$, of any
$f\in H^1(\rn)$ and $g\in\bbmo(\rn)$, denoted by $f\times g$,  can be
written as $$f\times g=S(f,\,g)+T(f,\,g);$$
Furthermore, there exists a positive constant $C$ such that, for any $(f,g)\in  H^1(\rn)\times \bbmo(\rn)$,
$$\|S(f, g)\|_{L^1(\rn)}\le C \|f\|_{H^1(\rn)}\|g\|_{\bbmo(\rn)}$$
and
$$\|T(f, g)\|_{H^{\log}(\rn)}\le C \|f\|_{H^1(\rn)}\lf[\|g\|_{\bbmo(\rn)}
+\lf|\fint_{B(\vec 0_n, 1)} g(x)\,dx\r|\r],$$
here and hereafter, $\vec 0_n$ denotes the \emph{origin} of $\rn$ and $B(\vec 0_n, 1)$
the \emph{open unit ball} of $\rn$ at $\vec 0_n$.
Moreover, it was proved in \cite{BGK12} that the space $H^{\log}(\rn)$ is optimal
in the sense that it can not be replaced by a smaller vector space.
The above result of \cite{BGK12} also answers a conjecture raised in \cite{BIJZ07}.
Based on this result, Ky \cite{Ky14} further developed a general real-variable theory of
Musielak-Orlicz Hardy spaces (see also \cite{ylk17} for a complete survey).
Thus, the space $H^{\log}(\rn)$ plays a role as a prototype
in the study of the real-variable theory of general Musielak-Orlicz Hardy
spaces.

Recently, the result of \cite{BGK12} was extended to the case $p\in(0,1)$ in \cite{BCKLYY}.
Indeed, when $p\in(0,1)$, it was proved in \cite{BCKLYY} that the space $H^{\Phi_p}(\rn)$ is
the optimal function space that is adapted to the bilinear
decomposition of the product of elements from the Hardy space $H^p(\rn)$ and
its dual space  $\mathfrak{C}_{1/p-1}(\rn)$
\begin{align}\label{eq-Hp}
H^p(\rn)\times\mathfrak{C}_{1/p-1}(\rn)\subset L^1(\rn)+H^{\Phi_p}(\rn).
\end{align}
The precise meaning of \eqref{eq-Hp} is as following: there exist two bounded bilinear operators
$S:\, H^p(\rn)\times \mathfrak C_{1/p-1}(\rn)\to L^1(\rn)$ and
$T:\, H^p(\rn)\times \mathfrak C_{1/p-1}(\rn)\to H^{\Phi_p}(\rn)$
such that the product, defined in the sense of $\cs'(\rn)$, of any  $f\in H^p(\rn)$
and $g\in\mathfrak C_{1/p-1}(\rn)$,  denoted by $f\times g$,  can be
written
as $$f\times g=S(f,\,g)+T(f,\,g)$$
and, furthermore, there exists a positive constant $C$ such that,
for any $(f,g)\in  H^p(\rn)\times \mathfrak C_{1/p-1}(\rn)$,
$$\|S(f, g)\|_{L^1(\rn)}\le C \|f\|_{H^p(\rn)}\|g\|_{\mathfrak C_{1/p-1}(\rn)}$$
and
$$\|T(f, g)\|_{H^{\Phi_p}(\rn)}\le C \|f\|_{H^p(\rn)}\lf[\|g\|_{\mathfrak C_{1/p-1}(\rn)}+\fint_{B(\vec 0_n, 1)} |g(x)|\,dx\r].$$
In particular, the target space $H^{\Phi_p}(\rn)$ in \eqref{eq-Hp}
was proved, in \cite{BCKLYY}, optimal
in the sense that it can not be replaced by a smaller vector space.

It should be mentioned that the study of the bilinear
decomposition of the product of elements from the Hardy space $H^p(\rn)$ and its dual space
can help us to improve the boundedness of many nonlinear qualities such as the div-curl product and
the weak Jacobian (see \cite{CLMS93,BGK12, BFG}) as well as
the endpoint boundedness of commutators (see \cite{Ky13, LCFY}).

Motivated by the aforementioned results of
\cite{BGK12,Ky14,BCKLYY}, it is interesting  to give a better understanding
of the structure of the Musielak-Orlicz Hardy
space $H^{\Phi_p}(\rn)$.
It is easy to observe that, for any $x\in\rn$ and $t\in(0,\infty)$, one has
\begin{align}\label{Phip-1}
\Phi_{p}(x,\,t)\sim
\begin{cases}
\dfrac{t}{1+[t(1+|x|)^n]^{1-p}}& \ \text{when}\ n(1/p-1)\notin \nn\cup\{0\},\vspace{0.2cm}\\
\dfrac{t}{1+[t(1+|x|)^n]^{1-p}[\log(e+|x|)]^p}& \ \text{when}\ n(1/p-1)\in \nn,\vspace{0.2cm}\\
\dfrac{t}{\log(e+t)+\log(e+|x|)}& \ \text{when}\ p=1
\end{cases}
\end{align}
with the equivalent positive constants independent of $x$ and $t$.
Based on \eqref{Phip-1},  for any $x\in\rn$ and $t\in[0,\,\fz)$,
we consider the Orlicz function
\begin{align}\label{fai}
\phi_0(t):=\frac{t}{\log(e+t)}
\end{align}
and the weight function
\begin{align}\label{Wp}
W_{p}(x):=
\begin{cases}
\dfrac{1}{\lf(1+|x|\r)^{n(1-p)}}& \ \text{when}\ n(1/p-1)\notin \nn\cup\{0\},\\
\dfrac{1}{\lf(1+|x|\r)^{n(1-p)}\lf[\log(e+|x|)\r]^p}& \ \text{when}\
n(1/p-1)\in \nn,\\
\dfrac{1}{\log(e+|x|)}& \ \text{when}\ p=1.
\end{cases}
\end{align}
Let $H^{\phi_0}(\rn)$  and
$H_{W_p}^p(\rn)$ be respectively
the {\it Orlicz-Hardy space} associated with $\phi_0$ and the
{\it weighted Hardy space} associated with $W_p$, which are defined in the same way
as Definition \ref{defn-Hp}(ii), but with $\|f\|_{L^{\Phi_p}(\rn)}$ therein replaced respectively by
\begin{align*}
\lf\|f\r\|_{L^{{\phi_0}}(\rn)}
:=\inf\lf\{\lz\in(0,\infty):\, \int_\rn {\phi_0}(|f(x)|/\lz)\,dx\le 1\r\}
\end{align*}
and
\begin{align*}
\lf\|f\r\|_{L_{W_p}^p(\rn)}:=\lf\{\dint_\rn \lf|f(x)\r|^pW_p(x)\,dx\r\}^{1/p}.
\end{align*}
We refer the reader to \cite{Ga79, Str79,Ja80,ST89} for more
properties on general Orlicz-Hardy spaces
and weighted Hardy spaces.

Recall that, in \cite{BL76}, for any two quasi-Banach spaces $A_0$ and $A_1$,
the pair $(A_0,\,A_1)$ is said to be {\it compatible} if there exists a Hausdorff space
$\mathbb{X}$ such that $A_0\subset \mathbb{X}$ and $A_1\subset \mathbb{X}$. For
any compatible pair $(A_0,\,A_1)$ of quasi-Banach spaces, the {\it sum space} $A_0+A_1$
is defined by setting
\begin{align}\label{sumsp}
A_0+A_1:=\lf\{a\in \mathbb{X}:\ \exists\, a_0\in A_0\,\, \text{and}\, \,a_1\in A_1\,\text{such that}\,\,
\,a=a_0+a_1\r\}
\end{align}
equipped with the {\it quasi-norm}
\begin{eqnarray*}
\|a\|_{A_0+A_1}:=\dinf\lf\{\|a_0\|_{A_0}+\|a_1\|_{A_1}:\, a=a_0+a_1,\, a_0\in A_0 \,\, \text{and}\,\,
a_1\in A_1\r\}.
\end{eqnarray*}
In what follows, we use $H^{\phi_0}(\rn)+H_{W_p}^p(\rn)$ (resp., $H^1(\rn)+H_{W_p}^p(\rn)$)
to denote the sum space, defined as in \eqref{sumsp}, with $\mathbb{X}:=\mathcal{S}'(\rn)$,
$A_0:=H^{\phi_0}(\rn)$ (resp., $A_0:=H^1(\rn)$) and $A_1:=H_{W_p}^p(\rn)$.

The main result of this article is the following representation of
$H^{\Phi_p}(\rn)$ for any $p\in(0,1]$ as the sum
of an (Orlicz-)Hardy and a weighted Hardy spaces.

\begin{theorem}\label{thm1}
Let $p\in(0,\,1]$. Define
$\Phi_p$,  ${\phi_0}$ and ${W_p}$
as in \eqref{Phip}, \eqref{fai} and \eqref{Wp}, respectively. Then
\begin{enumerate}
\item[{\rm(i)}]
the space $H^{\Phi_p}(\rn)$ and $H^{\phi_0}(\rn)+H_{W_p}^p(\rn)$ coincide
with equivalent quasi-norms;

\item[{\rm(ii)}] for any $p\in(0,\,1)$, the space
$H^{\Phi_p}(\rn)$ and $H^1(\rn)+H_{W_p}^p(\rn)$ coincide
with equivalent quasi-norms.
\end{enumerate}
\end{theorem}

The new structure of the Musielak-Orlicz Hardy space $H^{\Phi_p}(\rn)$ established in Theorem \ref{thm1}
enables us to reduce the study of many properties of $H^{\Phi_p}(\rn)$
to the corresponding ones of the Orlicz-Hardy space $H^{\phi_0}(\rn)$ when $p\in(0,\,1]$
(or the Hardy space $H^1(\rn)$ when $p\in(0,\,1)$)
and the weighted Hardy space $H_{W_p}^p({\rn})$, where the latter three
kinds of Hardy-type spaces are well studied
in various literatures; see, for example, \cite{FS72,Lu95,Ga79,Str79,Ja80,ST89}
and their references. A major job in the proof of Theorem \ref{thm1} is
decomposing every $f\in H^{\Phi_p}(\rn)$ into the sum of two parts, which
belongs to the desired sum space. We obtain this decomposition
by using the atomic characterization of $H^{\Phi_1}(\rn)$
when $p=1$ and the Calder\'on-Zygmund decomposition
of $H^{\Phi_p}(\rn)$ when $p\in(0,\,1)$.
The main trick is that we use different selection principles
in different decompositions and these selection principles
are based on the norm estimates for the characteristic functions of
the balls, which are established in Section \ref{s2}.

As an application of Theorem \ref{thm1}, we consider a concrete problem
of the interpolation of quasilinear operators.
Recall that the following definition of quasilinear operators is from \cite{Gr08}.
Let $T$ be an operator defined on some quasi-Banach space $A$ and taking values
in the set of all complex-valued finite almost everywhere measurable functions on $\rn$.
Such an operator $T$ is said to be  {\it quasilinear} if there exists a positive constant
$C$ such that, for any $f, \, g\in {A}$ and $\lz\in\cc$,
\begin{align*}
\lf|T(f)\r|= |\lz|\lf|f\r| \ \ \text{and} \ \
\lf|T(f+g)\r|\le C \lf(\lf|f\r|+\lf|g\r|\r).
\end{align*}

\begin{theorem}\label{thm2}
Let $p\in(0, \,1]$. Let $\Phi_p$, ${\phi_0}$ and $W_p$  be
as in \eqref{Phip}, \eqref{fai} and \eqref{Wp}, respectively. Assume that $T$ is a
quasilinear operator bounded on $H_{W_p}^p(\rn)$. Then
\begin{enumerate}
\item[{\rm(i)}] if T is bounded on $H^{\phi_0}(\rn)$, then
$T$ is bounded on $H^{\Phi_p}(\rn)$;

\item[{\rm(ii)}] if $p\in(0,\,1)$ and $T$ is bounded on $H^1(\rn)$, then
$T$ is bounded on $H^{\Phi_p}(\rn)$.
\end{enumerate}
\end{theorem}

This article is organized as follows. In Section \ref{s2}, we establish several
technical lemmas which are needed in the proof of Theorems
\ref{thm1} and \ref{thm2}. Section \ref{s3} is devoted to the proof of Theorem
\ref{thm1}. Finally, using Theorem
\ref{thm1}, we prove Theorem \ref{thm2} in Section \ref{s4}.

At the end of this section, we make some convention on the notation. Throughout this article, let $\nn:=\{1,2,\ldots\}$,
$\zz_+:=\nn\cup\{0\}$ and $\zz:=\{0,\pm1, \pm2,\dots\}$.
For any $x\in\rr^n$ and $r\in(0,\fz)$, denote by $B(x, r)$ the ball with center $x$ and radius $r$, that is, $B(x,r):=\{y\in\rr^n:\ |x-y|<r\}$.
For any ball $B\subset\rn$, we always denote by $c_B$ its center  and $r_B$ its radius
and, for any $\lz\in (0,\fz)$, by $\lz B$ the ball with center $c_B$ and radius $\lz r_B$.
For any set $E\subset\rn$, $\chi_E$ denotes its
{\it characteristic function}.
We use $C$ to denote a {\it positive
constant} that is independent of the main parameters involved but
whose value may differ from line to line.
Constants with subscripts, such as $C_1$, do not
change in different occurrences. If $f\le Cg$, we then write $f\ls
g$ and, if $f\ls g\ls f$, we then write $f\sim g$.
For any $s\in\rr$, let $\lfloor s\rfloor$ be the largest integer not greater than $s$.
We always use $\az$ to denote a multi-index $(\az_1,\dots, \az_n)$ with every $\az_i$
being a non-negative integer.

\section{Several technical lemmas\label{s2}}

\hskip\parindent In this section, we present several technical
lemmas which serve as preparations to prove Theorems \ref{thm1}
and \ref{thm2}.
To this end, we begin with recalling some notions used in \cite{Ky14}.

\begin{definition}\label{17-defOF}
For any $p\in(0,\infty)$, an \emph{Orlicz function} $\phi$ (which means that $\phi$ is nondecreasing and satisfies $\phi(0)=0$,
$\phi(t)>0$ for $t\in(0,\infty)$ and $\lim_{t\to\infty}\phi(t)=\infty$) is said to
be of {\it positive lower} (resp., {\it upper}) {\it type} $p$ if there exits a positive constant
$C$ such that, for any $t\in[0,\infty)$ and $s\in(0,1]$ (resp., $s\in[1,\infty)$),
$$
\phi(st)\le C s^p\phi(t).
$$
\end{definition}

\begin{definition}\label{17-defMOF}
A function $\phi:\, \rn\times [0,\infty)\to [0,\infty)$ is called a {\it Musielak-Orlicz function} if
 the function $\phi(x,\cdot):\, [0,\infty)\to [0,\infty)$ is an Orlicz function for any $x\in\rn$,
 and the function $\phi(\cdot, t)$ is a measurable function for any $t\in [0,\infty)$.
\end{definition}

\begin{definition}\label{17-defUT}
Let $\phi$ be a Musielak-Orlicz function.
For any given $p\in(0,\infty)$, the function $\phi$ is said to be of {\it positive uniformly
lower} (resp., {\it upper}) {\it type} $p$ if there exits a positive constant
$C$ such that, for any $x\in\rn$, $t\in[0,\infty)$ and $s\in(0,1]$ (resp., $s\in[1,\infty)$),
$$
\phi(x, st)\le C s^p\phi(x, t).
$$
\end{definition}

\begin{definition}\label{17-defUMC}
Let $\phi$ be a Musielak-Orlicz function and $q\in[1,\infty)$.
The function $\phi$ is said to satisfy the
{\it uniformly Muckenhoupt $\mathbb A_q(\rn)$ condition},
namely,  $\phi\in \mathbb A_q(\rn)$, if
\begin{align*}
[\phi]_{\mathbb A_q(\rn)}
:=\hspace{-0.1cm}
\begin{cases}
\displaystyle\sup_{t\in(0,\infty)}
\sup_{B\subset\rn} \lf[\frac1{|B|}\int_B \phi(x,t)\,dx\r]
\lf[\frac1{|B|}\int_B \{\phi(x,t)\}^{\frac{-1}{q-1}}\,dx\r]^{q-1}
\hspace{0.4cm}&\text{when}\; q\in(1,\infty),\vspace{0.2cm}\\
\displaystyle\sup_{t\in(0,\infty)}\sup_{B\subset\rn} \lf[\frac1{|B|}\int_B \phi(x,t)\,dx\r]
\lf[\mathop{\sup}_{x\in B} \{\phi(x,t)\}^{-1}\r] &\text{when}\; q=1
\end{cases}
\end{align*}
is finite, where the second suprema are taken over all balls $B$ of $\rn$.
Let
$$\mathbb A_\infty(\rn):= \bigcup_{q\in[1,\infty)}\mathbb A_q(\rn).$$
\end{definition}

\begin{remark} \label{rem-new}
Let  $p\in(0,1]$, $\Phi_p$ be as in \eqref{Phip}, $\phi_0$ as in \eqref{fai} and $W_p$ as in \eqref{Wp}.
\begin{enumerate}
\item[\rm (i)]
We know (see \cite{Ky14} for the case $p=1$ and
\cite{BCKLYY} for the case $p\in(0,1)$) that  $\Phi_p$  is a Musielak-Orlicz function
of uniformly upper
$1$ and of uniformly lower type $p$, and
belongs to the uniformly Mukenhoupt weight class
$\mathbb{A}_1(\rn)$.

\item[\rm (ii)]
Notice that $\phi_0(t)\sim \Phi_1(0,\,t)$ and
$W_p(x)\sim \Phi_p(x,\,1)$, where the equivalent positive constants are
independent of $x$ and $t$. From these and (i) of this remark, it follows immediately that $\phi_0$ is of upper and lower types $1$,
and $W_p$ belongs
to the usual Muckenhoupt weight class $A_1(\rn)$. In particular,
there exists a positive constant $C$ such that,
for any ball $B$ in $\rn$,
\begin{align}\label{f-A1}
\frac{1}{|B|}\dint_{B}W_p(x)\,dx
\le C \mathop{\mathrm{essinf}}_{y\in B} W_p(y).
\end{align}

\end{enumerate}

\end{remark}

For any $p\in(0,\,1]$, the next lemma  provides an $L^{\Phi_p}(\rn)$-norm
estimate for $\chi_B$, which was proved in \cite{Ky14} when $p=1$ and \cite{BCKLYY} when $p\in(0,1)$.

\begin{lemma}\label{lem-mon}
Let $p\in(0,\,1]$, $\az=1/p-1$ and $B=B(c_B,\,r_B)$ with $c_B\in\rn$ and $r_B\in(0,\infty)$.

\begin{enumerate}

\item[{\rm(i)}] If $p=1$,  then
$$\|\chi_B\|_{L^{\Phi_1}(\rn)}\sim
\frac{|B|}{\log(e+1/|B|)+\sup_{x\in B}\lf[\log(e+|x|)\r]}
\sim \frac{|B|}{|\log r_B|+\log(e+|c_B|)},$$
where the equivalent constants are positive and independent of $B$.

\item[{\rm(ii)}] If $p\in(0,\,1)$, then
$$\|\chi_B\|_{L^{\Phi_p}(\rn)}\sim \Psi_\alpha (B)|B|,$$
where
$$\Psi_\alpha (B)
:=\begin{cases}
\min\lf\{1, \lf(\dfrac{r_B}{1+|c_B|}\r)^{n\alpha}\r\}&\qquad \textup{when}\; n\alpha\notin \nn,
\vspace{0.2cm}\\
\min\lf\{1, \lf(\dfrac{r_B}{1+|c_B|}\r)^{n\alpha}\r\}\,\dfrac{1}{\log(1+|c_B|+r_B)}&\qquad \textup{when}\; n\alpha\in \nn
\end{cases}
$$
and the equivalent constants are positive and independent of $B$.
\end{enumerate}
\end{lemma}

For any $p\in(0,\,1]$ and any ball $B\subset\rn$, we still need to consider the Orlicz norm $\|\cdot\|_{L^{\phi_0}(\rn)}$ and  the weighted Lebesgue norm
$\|\cdot\|_{L_{W_p}^{p}(\rn)}$
estimates of the characteristic function $\chi_B$.

\begin{lemma}\label{lem-on}
Let $\phi_1$ be as in \eqref{Phip}, $W_1$ as in \eqref{Wp}, and ${\phi_0}$ as in \eqref{fai}. Define ${\phi_0}^{-1}$ to be
the inverse function of ${\phi_0}$.
For any $t\in(0,\,\fz)$, let
\begin{align}\label{f-rho}
\rho(t):=\frac{t^{-1}}{{\phi_0}^{-1}(t^{-1})}.
\end{align}
Then, for any ball $B\subset\rn$,
\begin{enumerate}
\item[{\rm (i)}]
$\|\chi_B\|_{L^{\phi_0}(\rn)}=|B|\rho(|B|)\sim \frac{|B|}{\log(e+{1}/{|B|})}$;

\item[{\rm (ii)}] $\|\chi_B\|_{L_{W_1}^1(\rn)}\sim
\frac{|B|}{\sup_{x\in B}\lf[\log(e+|x|)\r]}$,

\item[{\rm (iii)}] $
\|\chi_B\|^{-1}_{L^{\Phi_1}(\rn)}\sim
\|\chi_B\|^{-1}_{L^{\phi_0}(\rn)}+\|\chi_B\|^{-1}_{L_{W_1}^{1}(\rn)},
$
\end{enumerate}
where the equivalent constants in (i), (ii) and (iii) are positive and independent of $B$.
\end{lemma}

\begin{proof}
We first prove (i). Recall that the equivalence
$\|\chi_B\|_{L^{\phi_0}(\rn)}\sim \frac{|B|}{\log(e+{1}/{|B|})}$
was established in \cite[Lemma 7.13]{YY12}. Thus, to finish the proof of (i), it remains
to establish the first equality of (i). Indeed,
from the definition of $\rho$, it follows that
\begin{align*}
\dint_{\rn}{\phi_0}\lf(\frac{\chi_B(x)}{|B|\rho(|B|)}\r)\,dx
= \dint_{\rn}{\phi_0}\lf({\phi_0}^{-1}\lf(|B|^{-1}\r) \chi_B(x)\r)\,dx
=\int_B {\phi_0}\lf({\phi_0}^{-1}\lf(|B|^{-1}\r)\r)\,dx=1,
\end{align*}
which immediately implies the first equality of (i) and hence (i) holds true.

We now prove (ii). By the fact that $W_1$ belongs to the Muckenhoupt
weight class $A_1(\rn)$ (see \eqref{f-A1}), we know that
\begin{align*}
\|\chi_B\|_{L_{W_1}^1(\rn)}
=|B|\lf[\frac{1}
{|B|}\dint_{B}W_1(x)\,dx\r]\sim |B|\dinf_{x\in B} W_1(x)\sim |B|\dinf_{x\in B} \lf[\frac{1}{\log(e+|x|)}\r]\sim \frac{|B|}
{\sup_{x\in B} \log(e+|x|)},
\end{align*}
which implies that (ii) holds true.

Finally,  (iii) follows directly from  Lemma \ref{lem-mon}(i) and (i) and (ii) of this lemma.
This finishes the proof of Lemma \ref{lem-on}.
\end{proof}


\begin{lemma}\label{lem-whn}
Let $p\in(0,\,1)$, $\Phi_p$ and $W_p$ be
respectively as in \eqref{Phip} and \eqref{Wp}. Then, for any ball $B\subset\rn$,
$$\|\chi_B\|_{L_{W_p}^p(\rn)}\sim\|\chi_B\|_{L^{\Phi_p}(\rn)},$$
where the equivalent constants  are positive and independent of $B$.
\end{lemma}

\begin{proof}
Denote by $c_B$ the center of $B$ and $r_B$ its radius.
By Lemma \ref{lem-mon}(ii),
we have
\begin{align}\label{f-whn1}
\|\chi_B\|_{L^{\Phi_p}(\rn)}
\sim
\begin{cases}
|B| \min\lf\{1,\,\lf(\dfrac{r_B}{1+|c_B|}\r)^{n(1/p-1)}\r\}\quad&\textup{when}\; n(1/p-1)\notin\nn,\vspace{0.2cm}\\
|B| \min\lf\{1,\,\lf(\dfrac{r_B}{1+|c_B|}\r)^{n(1/p-1)}\r\}\dfrac1{\log(e+r_B+|c_B|)}
\quad&\textup{when}\; n(1/p-1)\in\nn.
\end{cases}
\end{align}
To estimate $\|\chi_B\|_{L_{W_p}^p(\rn)}$, we consider the following three cases.

\underline{\it Case (i): $|c_B|\ge 2r_B$.} In this case,
for any $x\in B$, it is easy to see that
$|x|\sim |c_B|$. By this and \eqref{f-whn1}, we conclude that, when $ n(1/p-1)\notin\nn$,
\begin{align*}
\lf\|\chi_B\r\|_{L_{W_p}^p(\rn)}
&=\lf\{\dint_B\frac{1}{(1+|x|)^{n(1-p)}}\,dx\r\}^{1/p}\\
&\sim \lf\{\dint_B\frac{1}{(1+|c_B|)^{n(1-p)}}\,dx\r\}^{1/p}
\sim \frac{|B|^{1/p}}{(1+|c_B|)^{n(1/p-1)}}\sim \|\chi_B\|_{L^{\Phi_p}(\rn)},
\end{align*}
as desired. Similarly, when $n(1/p-1)\in\nn$, we have
\begin{align*}
\lf\|\chi_B\r\|_{L_{W_p}^p(\rn)}
&=\lf\{\dint_B\frac{1}{(1+|x|)^{n(1-p)}\lf[\log(e+|x|)\r]^p}\,dx\r\}^{1/p}\\
&
\sim \frac{|B|^{1/p}}{(1+|c_B|)^{n(1/p-1)}\log(e+|c_B|)}\sim \|\chi_B\|_{L^{\Phi_p}(\rn)}.
\end{align*}

\underline{\it Case (ii): $|c_B|< 2r_B<1$.}
In this case, for any $x\in B$, we have
$|x|\le |x-c_B|+|c_B|<r_B+|c_B|<2.$
Thus, whenever  $n(1/p-1)$ is an integer or not, we always have
$$\dinf_{x\in B}W_p(x)
\sim 1.$$
From this and the fact that $W_p\in A_1(\rn)$ (see \eqref{f-A1}),
it follows that
\begin{align*}
\lf\|\chi_B\r\|_{L_{W_p}^p(\rn)}&
=|B|^{1/p}\lf[\frac{1}{|B|}\dint_BW_p(x)\,dx\r]^{1/p}\sim |B|^{1/p}\lf[\dinf_{x\in B}W_p(x)\r]^{1/p}
\sim |B|^{1/p}\sim \|\chi_B\|_{L^{\Phi_p}(\rn)},
\end{align*}
as desired.

\underline{\it Case (iii): $|c_B|< 2r_B$ and $2r_B\ge1$.}
In this case,  for any
$x\in B$, we have  $|x|\le |x-c_B|+|c_B|<r_B+|c_B|<3r_B$, so that $1+|x|\ls r_B$ and hence
\begin{align*}
\inf_{x\in B} W_p(x)
&=\begin{cases}
\displaystyle\inf_{x\in B} \dfrac1{(1+|x|)^{n(1-p)}}\quad&\textup{when}\; n(1/p-1)\notin\nn,\vspace{0.2cm}\\
\displaystyle\inf_{x\in B} \dfrac1{(1+|x|)^{n(1-p)}[\log(e+|x|)]^p}
\quad&\textup{when}\; n(1/p-1)\in\nn
\end{cases}\\
&\sim
\begin{cases}
 \dfrac1{|B|^{1-p}}\quad&\textup{when}\; n(1/p-1)\notin\nn,\vspace{0.2cm}\\
 \dfrac1{|B|^{1-p}[\log(e+r_B)]^p}
\quad&\textup{when}\; n(1/p-1)\in\nn.
\end{cases}
\end{align*}
Consequently,
\begin{align*}
\lf\|\chi_B\r\|_{L_{W_p}^p(\rn)}
&=|B|^{1/p}\lf[\frac{1}{|B|}\dint_BW_p(x)\,dx\r]^{1/p}
\sim |B|^{1/p}\lf[\dinf_{x\in B}W_p(x)\r]^{1/p}\\
&\gs \begin{cases}
 |B|\quad&\textup{when}\; n(1/p-1)\notin\nn,\vspace{0.2cm}\\
 \dfrac{|B|}{\log(e+r_B)}
\quad&\textup{when}\; n(1/p-1)\in\nn
\end{cases}\\
&
\sim  \|\chi_B\|_{L^{\Phi_p}(\rn)}.
\end{align*}
Also, when $n(1/p-1)\notin\nn$, we have
\begin{align*}
\lf\|\chi_B\r\|_{L_{W_p}^p(\rn)}
\le \lf[\int_{|x|<3r_B}\frac{1}{(1+|x|)^{n(1-p)}}\,dx\r]^{1/p}\sim \frac{|B|^{1/p}}{(1+3r_B)^{n(1/p-1)}}
\sim |B|\sim \|\chi_B\|_{L^{\Phi_p}(\rn)}.
\end{align*}
Meanwhile, when  $n(1/p-1)\in\nn$, we obtain
\begin{align*}
\lf\|\chi_B\r\|_{L_{W_p}^p(\rn)}
&\le \lf[\int_{|x|<3r_B}\frac{1}{(1+|x|)^{n(1-p)}[\log(e+|x|)]^p}\,dx\r]^{1/p}\\
&= \lf[\sum_{j=1}^\infty \int_{2^{-j}3r_B\le |x|<2^{-j+1}3r_B}\frac{1}{(1+|x|)^{n(1-p)}[\log(e+|x|)]^p}\,dx\r]^{1/p}\\
&\ls \lf[\sum_{j=1}^\infty \frac{(2^{-j} r_B)^n}{(1+2^{-j} r_B)^{n(1-p)}[\log(e+2^{-j} r_B)]^p}\r]^{1/p}\\
& \ls \lf[\sum_{j=1}^\infty \frac{j^p(2^{-j} r_B)^{np}}{[\log(e+ r_B)]^p} \r]^{1/p}
\ls \dfrac{|B|}{\log(e+r_B)}
\sim \|\chi_B\|_{L^{\Phi_p}(\rn)},
\end{align*}
where we used the following estimates:
$$
\frac{\log(e+ r_B)}{\log(e+2^{-j} r_B)}\le \log(e+2^{j})\ls j.
$$
Altogether, we find that $\lf\|\chi_B\r\|_{L_{W_p}^p(\rn)}\sim \|\chi_B\|_{L^{\Phi_p}(\rn)}$ in the case
$|c_B|< 2r_B$ and $2r_B\ge1$.

Summarizing the above three cases, we conclude that (ii) holds true.
This finishes the proof of Lemma \ref{lem-whn}.
\end{proof}

From Lemmas \ref{lem-mon}, \ref{lem-on}  and \ref{lem-whn},
we deduce some interesting properties on the
Musielak-Orlicz Hardy space $H^{\Phi_p}(\rn)$ for any $p\in(0,1]$.
To be precise, we first recall the following definition
of $H^{\Phi_p}$-atoms from \cite{Ky14, LY13}.

\begin{definition}\label{d5.4}
Let $p\in(0,\,1]$, $\Phi_p$ be as in \eqref{Phip}, $q\in(1,\fz]$ and $s\in\zz_+\cap[\lfz
n(1/p-1)\rfz,\,\fz)$.
\begin{enumerate}
\item[\rm(I)] For each ball $B\subset\rn$, the \emph{space} $L^q_{\Phi_p}(B)$ with
$q\in[1,\fz]$ is defined to be the set of all measurable functions
$f$ on $\rn$, supported in $B$, such that
\begin{align*}
\|f\|_{L^q_{\Phi_p}(B)}:=
\begin{cases}\dsup_{t\in (0,\fz)}
\lf[\frac{1}
{{\Phi_p}(B,t)}\dint_{\rn}|f(x)|^q\Phi_p(x,t)\,dx\r]^{1/q},& q\in [1,\fz);\\
\|f\|_{L^{\fz}(B)},&q=\fz
\end{cases}
\end{align*}
is finite,
where, for any measurable set $E$ and $t\in[0,\,\fz)$,
$
{\Phi_p}(E,t):=\int_{E}\Phi_p(x,\,t)\,dx.
$

\item[\rm(II)] A function $a$  is called a \emph{$(\Phi_p,\,q,\,s)$-atom}
if there exists a ball
$B\subset\rn$ such that

\begin{enumerate}
\item[\rm(i)] $\supp a\subset B$;

\item[\rm(ii)]
$\|a\|_{L^q_{\Phi_p}(B)}\le\|\chi_B\|_{L^{\Phi_p}(\rn)}^{-1}$;

\item[\rm(iii)] $\int_{\rn}a(x)x^{\az}\,dx=0$ for all
$\az\in\zz_+^n$ with $|\az|\le s$.
\end{enumerate}

\item[\rm(III)] The  \emph{atomic Musielak-Orlicz Hardy space}
$H^{\Phi_p,\,q,\,s}(\rn)$ is defined to be the space of all
$f\in\cs'(\rn)$ satisfying that $f=\sum_j\lz_ja_j$ in $\cs'(\rn)$,
where $\{\lz_j\}_j\subset\cc$ and  $\{a_j\}_j$ is a sequence of
$(\Phi_p,\,q,\,s)$-atoms, respectively, associated with balls $\{B_j\}_j$, satisfying
$$\sum_j\Phi_p\lf(B_j,\,\frac{|\lz_j|}{\|\chi_B\|_{L^{\Phi_p}(\rn)}}\r)<\fz.$$
Moreover, let
\begin{align*}
\blz_{\Phi_p}(\{\lz_j a_j\}_j):= \inf\lf\{\lz\in(0,\fz):\ \ \sum_j\Phi_p
\lf(B_j,\frac{|\lz_j|}{\lz\|\chi_{B_j}\|_{L^{\Phi_p}(\rn)}}\r)\le1\r\}.
\end{align*}
Then the \emph{quasi-norm} of $f\in H^{\Phi_p,\,q,\,s}(\rn)$
is defined by setting
\begin{align}\label{f-SN-1}
\|f\|_{H^{\Phi_p,\,q,\,s}(\rn)}:=\inf\lf\{
\blz_{\Phi_p}(\{\lz_ja_j\}_j)\r\},
\end{align}
where the infimum is taken over all the
decompositions of $f$ as above.
\end{enumerate}
\end{definition}

The following atomic characterization of $H^{\Phi_p}(\rn)$
follows from a general theory of the atomic characterization
of Musielak-Orlicz Hardy spaces established in \cite[Theorem~3.1]{Ky14}.

\begin{lemma}\label{l5.1}
Let $p\in(0,\,1]$, $\Phi_p$ be as in \eqref{Phip}, $q\in(1,\fz]$ and $s\in\zz_+\cap[\lfz
n(1/p-1)\rfz,\,\fz)$.  Then the spaces $H^{\Phi_p}(\rn)$ and $H^{\Phi_p,\,q,\,s}(\rn)$
coincide with equivalent quasi-norms.
\end{lemma}

\begin{remark}\label{rem-at}
Let $p\in(0,\,1]$, $q\in (1,\,\fz)$ and
$s\in\zz_+\cap[\lfz n(1/p-1)\rfz,\,\fz)$. Assume that
${\phi_0}$ and $W_p$ are as in
\eqref{fai} and \eqref{Wp}, respectively.
Following Definition \ref{d5.4}(II),
if we replace $\Phi_p(x,\,t)$ therein respectively by $t^p$, $t^p W_p(x)$ and ${\phi_0}(t)$,
then we obtain the definitions of \emph{$(p,\,q,\,s)$-atoms},
\emph{$(p,\,q,\,s)_{W_p}$-atoms}
and \emph{$({\phi_0},\,q,\,s)$-atoms}.
Correspondingly, we follow Definition \ref{d5.4}(III) to introduce the
atomic Hardy spaces
$H^{p,\,q,\,s}(\rn)$, $H_{W_p}^{p,\,q,\,s}(\rn)$
and
$H^{\phi_0,\,q,\,s}(\rn)$
by replacing
the quasinorm in \eqref{f-SN-1}, respectively, by
\begin{align*}
\|f\|_{H^{p,\,q,\,s}(\rn)}:=\inf\lf\{ \lf[\sum_{j\in\nn}|\lz_j|^p\r]^{1/p}\r\},
\end{align*}
\begin{align*}
\|f\|_{H_{W_p}^{p,\,q,\,s}(\rn)}:=\inf\lf\{ \lf[\sum_{j\in\nn}|\lz_j|^p\r]^{1/p}\r\}
\end{align*}
and
\begin{align*}
\|f\|_{H^{\phi_0,\,q,\,s}(\rn)}:=\inf\lf\{
\blz_{\phi_0}(\{\lz_ja_j\}_j)\r\},
\end{align*}
where
\begin{align*}
\blz_{{\phi_0}}(\{\lz_ja_j\}_j):= \inf\lf\{\lz\in(0,\fz):\ \ \sum_j
|B_j|\phi_0\lf(\frac{|\lz_j|}{\lz |B_j|\rho(B_j))}\r)\le1\r\}
\end{align*}
with $\rho$ as in \eqref{f-rho}. Then, from \cite[Theorem~3.1]{Ky14}, it also follows that
$$
\begin{cases}
H^p(\rn)= H^{p,\,q,\,s}(\rn)\\
H_{W_p}^p(\rn)=H_{W_p}^{p,\,q,\,s}(\rn)\\
H^{\phi_0}(\rn)= H^{\phi_0,\,q,\,s}(\rn)
\end{cases}
$$
with equivalent quasinorms. See also \cite{FS72,St93,Lu95,Vi87,JY10,Ga79,ST89}
and their references for more discussions on these three kinds of Hardy-type spaces.
\end{remark}

From these and Lemmas \ref{lem-mon}, \ref{lem-on}  and \ref{lem-whn}, we deduce the following proposition,
which is the basis
to prove Theorem \ref{thm1}.

\begin{proposition}\label{prop-new}
Let $p\in(0,\,1]$, $q\in (1,\,\fz)$
and $s\in\zz_+\cap[\lfz n(1/p-1)\rfz,\,\fz)$.
Let
$\Phi_p$,  ${\phi_0}$ and ${W_p}$
be as in \eqref{Phip}, \eqref{fai} and \eqref{Wp}, respectively.
Then, for any ball $B\subset \rn$ with center $c_B\in\rn$ and radius $r_B\in(0,\infty)$, the following assertions are true:
\begin{enumerate}
\item[\rm(i)] any $(\phi_0,\fz,s)$-atom or $(1,\fz,s)_{W_1}$-atom associated with the ball
$B$ is also a $(\Phi_1,\fz,s)$-atom associated with the same ball $B$;

\item[\rm(ii)] if $r_B<1$ and $|c_B|<1/r_B$, then any  $(\Phi_1,\fz,s)$-atom associated with the ball $B$
is also a $(\phi_0,\fz,s)$-atom associated with the same ball $B$;

\item[\rm(ii)] if $r_B<1$ and $|c_B|\ge 1/r_B$, or $r_B\ge 1$,
then any  $(\Phi_1,\fz,s)$-atom associated with the ball $B$
is also a $(1,\fz,s)_{W_1}$-atom associated with the same ball $B$;

\item[\rm (iv)]
when $p\in(0,1)$, any $(\Phi_p , \fz, s)$-atom associated with the ball $B$ is also a $(p, \fz, s)_{W_p}$-atom associated with
the same ball $B$, and vise versa.
\end{enumerate}

\end{proposition}

\begin{proof}
Let  $\chi_B$ be the characteristic function of the ball $B$.
By Lemma \ref{lem-on}(iii), we know that
$$\|\chi_B\|^{-1}_{L^{\phi_0}(\rn)}\ls \|\chi_B\|^{-1}_{L^{\Phi_1}(\rn)}$$
and
$$\|\chi_B\|^{-1}_{L_{W_1}^{1}(\rn)}\ls \|\chi_B\|^{-1}_{L^{\Phi_1}(\rn)}.$$
By these and the definitions of $(\Phi_1,\fz,\,s)$-atoms, $(\phi_0,\fz,s)$-atoms and $(1,\fz,s)_{W_1}$-atoms,
we know that any $(\phi_0,\fz,s)$-atom or $(1,\fz,s)_{W_1}$-atom associated with the ball
$B$ is also a $(\Phi_1,\fz,s)$-atom associated with $B$.
Hence, (i) holds true.

Now we show (ii). If $r_B<1$ and $|c_B|<1/r_B$, then, for any $x\in B$, it holds true that
$|x|\le |x-c_{B}|+|c_{B}|<1+\frac{1}{r_{B}}$, which implies that
\begin{align*}
\sup_{x\in B}\log(e+|x|)\le \log(e+1+1/r_{B})\sim \log(e+1/|B|).
\end{align*}
This, together with Lemmas \ref{lem-mon}(i) and \ref{lem-on}(i),
shows that
\begin{align*}
\lf\|\chi_{B}\r\|_{L^{\Phi_1}(\rn)}&\sim \frac{|B|}
{\log(e+1/|B|)+\sup_{x\in B}\log(e+|x|)}\sim \frac{|B|}
{\log(e+1/|B|)}\sim \lf\|\chi_{B}\r\|_{L^{\phi_0}(\rn)}.
\end{align*}
Then, by the definitions of $(\Phi_1,\fz,\,s)$-atoms and $(\phi_0,\fz,s)$-atoms,
we know that any  $(\Phi_1,\fz,s)$-atom associated with the ball $B$
is also a $(\phi_0,\fz,s)$-atom associated with  $B$. This finishes the proof of (ii).

To prove (iii),
 we claim that, if
$1/|c_B|\le r_B<1$  or $r_B\ge 1$,
then
\begin{align}\label{f-thm1-91}
\log\lf(e+1/|B|\r)+\sup_{x\in B}\log\lf(e+|x|\r)
\sim \sup_{x\in B}\log\lf(e+|x|\r).
\end{align}
Indeed, if $r_{B}\ge1$, then \eqref{f-thm1-91} holds true immediately.
If $1/|c_B|\le r_B<1$, then
$$
\log\lf(e+1/|B|\r)
\ls \log(e+|c_B|) \ls \sup_{x\in B}\log(e+|x|),
$$
whence leading to  \eqref{f-thm1-91}.  Thus, we conclude that
\begin{align*}
\lf\|\chi_{B}\r\|_{L^{\Phi_1}(\rn)}&\sim \frac{|B|}
{\log(e+1/|B|)+\sup_{x\in B}\log(e+|x|)}\sim \frac{|B|}
{\sup_{x\in B}\lf[\log\lf(e+|x|\r)\r]}\sim \lf\|\chi_{B}\r\|_{L_{W_1}^{1}(\rn)}.
\end{align*}
Then, applying the  definitions of $(\Phi_1,\fz,\,s)$-atom and $(1,\fz,s)_{W_1}$-atom, we see that any  $(\Phi_1,\fz,s)$-atom associated with the ball $B$
is also a $(1,\fz,s)_{W_1}$-atom associated with  $B$. This finishes the proof of (iii).

To show (iv), for any $p\in(0, 1)$, by the definitions of
$(\Phi_p , \fz, s)$-atoms and $(p, \fz, s)_{W_p}$-atoms
as well as Lemma \ref{lem-whn}, we immediately conclude that a function $a$ on
$\rn$ is a $(\Phi_p , \fz, s)$-atom associated with $B$ if and only
if a is a $(p, \fz, s)_{W_p}$-atom associated with the same ball $B$.
Thus, (iv) holds true, which completes the proof of Proposition \ref{prop-new}.
\end{proof}

We end this section by recalling
the following two lemmas, established in \cite{Ky14}, on the Calder\'on-Zygmund 
decomposition of the elements of Musielak-Orlicz Hardy spaces.

\begin{lemma}\label{lem-dp}
Let $p\in(0,\,1]$, $q\in(1,\,\fz)$ and $\Phi_p$
be as in \eqref{Phip}. Then $L_{\Phi_p(\cdot,1)}^q(\rn)\cap
H^{\Phi_p}(\rn)$ is dense in $H^{\Phi_p}(\rn)$.
\end{lemma}

\begin{lemma}\label{lem-CZd}
Let $p\in(0,\,1]$, $q\in(1,\,\fz)$,
$s\in \zz_+\cap [\lfloor n(1/p-1)\rfloor,\,\fz)$ and
$\Phi_p$ be as in \eqref{Phip}.
For any $f\in L_{\Phi_p(\cdot,1)}^q(\rn)\cap
H^{\Phi_p}(\rn)$, there exist  family $\{\Lambda_k\}_{k\in\zz}$ of index set with elements of
countable numbers,
$\{g_i^k\}_{k\in\zz,i\in\Lambda_k}$ and $\{b_i^k\}_{k\in\zz,i\in\Lambda_k} \subset \mathcal{S}'(\rn)$
such that
\begin{enumerate}

\item[{\rm(i)}] for any $k\in\zz$, $f=g^k+\sum_{i\in\Lambda_k}b_i^k$ in $\mathcal{S}'(\rn)$;

\item[{\rm(ii)}] $f=\dsum_{k\in\zz} (g^{k+1}-g^k)$ in $\mathcal{S}'(\rn)$;

\item[{\rm(iii)}] for any $k\in\zz$, there exists a family $\{b^k_i\}_{i\in\Lambda_k}
\subset L^\fz(\rn)$ such that $g^{k+1}-g^{k}=\sum_{i\in\Lambda_k} h^k_i$ in
$\mathcal{S}'(\rn)$;

\item[{\rm (iv)}]  for any $k\in\zz$ and $i\in\Lambda_k$, $h_i^k$ satisfies

\begin{enumerate}

\item [{\rm (a)}] $\supp h_i^k\subset B_i^k$, where $B_i^k:=18\wz B_i^k$
and  $\{\wz B_i^k\}_{i\in \Lambda_k}$ is a Whitney covering of $\Omega_k$ with
\begin{align*}
\Omega_k:=\lf\{x\in\rn:\ f^*(x)>2^k\r\},
\end{align*}
where $f^*$ denotes the non-tangential maximal function
of $f$ as in \eqref{grand-funct} with $m$ therein
equal to $\lfloor n(1/p-1)\rfloor$;

\item [{\rm (b)}] $\|h_i^k\|_{L^\fz(\rn)}\le c2^k$, where $c$ is a positive constant independent of
$k$, $i$ and $f$.

\item [{\rm (c)}] for any multi-index $\az$ satisfying
$|\az|\le s$, it holds true that
$
\int_{\rn} x^\az h_i^k(x)\,dx=0.
$

\end{enumerate}
\end{enumerate}

\end{lemma}

\section{Proof of Theorem \ref{thm1}}\label{s3}

\hskip\parindent
Based on the technical lemmas established in
Section \ref{s2}, we now prove Theorem \ref{thm1}
by considering  two cases: $p=1$ and $p\in(0, 1)$. We point out that  these two cases
are based on different selection principles to obtain the desired
sum space.

\begin{proof}[Proof of Theorem \ref{thm1} in the case $p=1$]
We first establish the inclusion
$H^{\phi_0}(\rn)+H_{W_1}^1(\rn)\subset H^{\Phi_1}(\rn)$.
For any $f\in H^{\phi_0}(\rn)+H_{W_1}^1(\rn)$, let $f_0\in H^{\phi_0}(\rn)$
and $f_1\in H_{W_1}^1(\rn)$ satisfy $f=f_0+f_1$ in $\mathcal{S}'(\rn)$ and
\begin{align*}
\|f\|_{H^{\phi_0}(\rn)+H_{W_1}^1(\rn)}\sim \|f_0\|_{H^{\phi_0}(\rn)}
+\|f_1\|_{H_{W_1}^1(\rn)}.
\end{align*}
Using \eqref{Phip}, we know that, for any $x\in\rn$ and $t\in(0,\,\fz)$,
\begin{align*}
\Phi_1(x,\,t)\ls
\min\lf\{\frac{t}{\log(e+t)},\,\frac{t}
{\log(e+|x|)}\r\},
\end{align*}
which, combined with the grand maximal function characterizations of these Hardy-type spaces, shows that
\begin{align*}
\|f\|_{H^{\Phi_1}(\rn)}&\ls \|f_0\|_{H^{\Phi_1}(\rn)}+\|f_1\|_{H^{\Phi_1}(\rn)}
\ls \|f_0\|_{H^{\phi_0}(\rn)}+\|f_1\|_{H_{W_1}^1(\rn)}
\sim \|f\|_{H^{\phi_0}(\rn)+H_{W_1}^1(\rn)}.
\end{align*}
This immediately implies the inclusion
$H^{\phi_0}(\rn)+H_{W_1}^1(\rn)\subset H^{\Phi_1}(\rn)$.

We now prove the converse inclusion
$H^{\Phi_1}(\rn)\subset H^{\phi_0}(\rn)+H_{W_1}^1(\rn)$. Without loss of generality, we may assume that
$f\in H^{\Phi_1}(\rn)$ and $\|f\|_{H^{\Phi_1}(\rn)}=1$; otherwise,
we use $\wz f:=f/\|f\|_{H^{\Phi_p}(\rn)}$ to replace $f$
in the same argument as below.

Let $s\in\zz_+$ and $s\ge n(1/p-1)$,  by Lemma \ref{l5.1},
we know that there exist $\{a_j\}_{j\in\nn}$ of $(\Phi_1,\,\fz,\,s)$-atoms and $\{\lz_j\}_{j\in\nn}\subset \cc$
such that
\begin{align}\label{f-thm1-1}
f=\dsum_{j\in\nn}\lz_j a_j
\end{align}
in $H^{\Phi_1}(\rn)$ and hence in $\mathcal{S}'(\rn)$, and
$\blz_{\Phi_1}(\{\lz_j a_j\}_j)\le 2$.
Since $\Phi_1$ of of uniformly lower type $1$ and of uniformly upper type $1$, it follows easily that, for any $x\in\rn$ and $s,t\in(0,\infty)$,
$$\Phi_1(x, st)\sim s\Phi_1(x,t).$$
By this, the fact $\blz_{\Phi_1}(\{\lz_j a_j\}_j)\le 2$ and \cite[Lemma~4.2(i)]{Ky14}, we conclude that
\begin{align}\label{eq-x1}
1\ge \sum_j\Phi_1
\lf(B_j,\frac{|\lz_j|}{2\|\chi_{B_j}\|_{L^{\Phi_1}(\rn)}}\r)
&=\sum_j\int_{B_j}\Phi_1\lf(x,\, \frac{|\lz_j|}{2\|\chi_{B_j}\|_{L^{\Phi_1}(\rn)}}\r)\,dx\\
&\sim \sum_j|\lz_j|\int_{B_j}\Phi_1\lf(x,\, \frac{1}{\|\chi_{B_j}\|_{L^{\Phi_1}(\rn)}}\r)\,dx\sim \sum_j|\lz_j|.\notag
\end{align}

For any $j\in\nn$, assume that  $a_j$ is supported on a ball $B_j:=B(c_j,\,r_{j})$, where $c_j\in\rn$ and $r_j\in(0,\infty)$.
Define
$$
{\rm I}_0:=\lf\{j:\, r_{j}<1\;\textup{and}\; |c_{j}|<\frac{1}{r_j}\r\}\quad
\textup{and}\quad
{\rm I}_1:=\lf\{j:\, r_{j}<1\;\textup{and}\; |c_{j}|\ge\frac{1}{r_j},\,\textup{or}\; r_j\ge 1\r\}.
$$
It is easy to see that
$\mathrm{I}_0\cap \mathrm{I}_1=\emptyset$.
We now write the decomposition in \eqref{f-thm1-1}
into
\begin{align}\label{f-thm1-1-1}
f&=\dsum_{j\in\nn}\lz_j a_j
=\dsum_{j\in\mathrm{I}_0}\lz_{j} a_{j}+\dsum_{j\in\mathrm{I}_1}\lz_{j} a_{j}
=:f_0+f_1.
\end{align}
Thus, by Proposition \ref{prop-new}, we know that, for any $j\in\mathrm{I}_0$,
$a_{j}$ is a $(\phi_0,\,\fz,\,s)$-atom associated with the ball $B_{j}$
and, for any $j\in \mathrm{I}_1$, $a_{j}$ is a $(1,\,\fz,\,s)_{W_1}$-atom associated with $B_{j}$.

We now  show $f_0\in H^{\phi_0}(\rn)$.
For any $j\in I_0$, by  $r_{j}<1$  and Lemma \ref{lem-on}(i), we know that
\begin{align*}
|{B_{j}}|{\phi_0}\lf(\frac{1}{|B_{j}|\rho(|B_{j}|)}\r)
&\sim |{B_{j}}|{\phi_0}\lf(\frac{\log(e+1/|B_{j}|)}{|B_{j}|}\r)
\sim \frac{\log(e+1/|B_{j}|)}{\log(e+\frac{\log(e+1/|B_{j}|)}{|B_{j}|})}
\ls 1,
\end{align*}
which, together with the fact that ${\phi_0}$ is of lower type $1$, further implies that
\begin{align*}
\dsum_{j\in\mathrm{I}_0}|B_{j}|
{\phi_0}\lf(\frac{|\lz_{j}|}{\sum_{j\in\mathrm{I}_0}|\lz_{j}||B_{j}|
\rho(|B_{j}|)}\r)
&\ls\dsum_{j\in\mathrm{I}_0}
\frac{|\lz_{j}|}
{\sum_{j\in\mathrm{I}_0}|\lz_{j}|}
|{B_{j}}|{\phi_0}\lf(\frac{1}{|B_{j}|\rho(|B_{j}|)}\r)\ls 1.
\end{align*}
From this and \eqref{eq-x1}, it follows that
\begin{align*}
\Lambda_{\phi_0}\lf(\{\lz_{j}a_{j}\}_{j\in\mathrm{I}_0}\r)
&=\dinf \lf\{\lz\in(0,\,\fz):\ \dsum_{j\in\mathrm{I}_0}|{B_{j}}|
{\phi_0}\lf(\frac{|\lz_{j}|}{\lz|B_{j}|\rho(|B_{j}|)}\r)\le 1\r\}
\ls\dsum_{j\in\mathrm{I}_0}\lf|\lz_{j}\r|\ls 1.
\end{align*}
Thus, $f_0\in H^{\phi_0}(\rn)$ and $\|f_0\|_{H^{\phi_0}(\rn)}\ls \Lambda_{\phi_0}\lf(\{\lz_{j}a_{j}\}_{j\in\mathrm{I}_0}\r)\ls 1\sim \|f\|_{H^{\Phi_1}(\rn)}$.

For $f_1$, using the atomic characterization of $H_{W_1}^1(\rn)$
stated in Remark \ref{rem-at}, \eqref{eq-x1} and \eqref{f-thm1-1-1},
we find that $f_1\in H_w^1(\rn)$ and
\begin{align*}
\lf\|f_1\r\|_{H_w^1(\rn)}&\ls \dsum_{j\in\mathrm{I}_1} \lf|\lz_{j}\r|
\ls 1\sim \lf\|f\r\|_{H^{\Phi_1}(\rn)}.
\end{align*} 

Thus, we conclude that for any $f\in H^{\Phi_1}(\rn)$,
there exist $f_0\in H^{\phi_0}(\rn)$ and $f_1\in H_{W_1}^1(\rn)$ such that
$f=f_0+f_1$ in $\mathcal{S}'(\rn)$ and
$\|f_0\|_{H^{\phi_0}(\rn)}+\|f_1\|_{H_{W_1}^1(\rn)}\ls \|f\|_{H^{\Phi_1}(\rn)}$.
This finishes the proof of the converse inclusion $H^{\Phi_1}(\rn)\subset H^{\phi_0}(\rn)+H_{W_1}^1(\rn)$
and hence   of Theorem \ref{thm1} in the case $p=1$.
\end{proof}

We now turn to the proof of Theorem \ref{thm1} under the case $p\in (0,\,1)$.

\begin{proof}[Proof of Theorem \ref{thm1} under the case $p\in(0,\,1)$.]
Let $p\in(0,\,1)$. For any $x\in\rn$ and $t\in(0,\,\fz)$, using \eqref{Phip} and \eqref{Phip-1}, we observe that
$$\Phi_{p}(x,\,t)\ls \min\lf\{\phi_0(t),\,{t^p}{W_{p}(x)}\r\}\qquad \textup{and}\qquad \phi_0(t)\le t.$$
From these observations, we argue as in the case $p=1$ and can obtain the inclusions
$$H^1(\rn)+H_{W_{p}}^p(\rn)\subset H^{\phi_0}(\rn)+H_{W_{p}}^p(\rn)\subset H^{\Phi_{p}}(\rn),$$
with desired norm estimates.

It remains to prove $H^{\Phi_p}(\rn)\subset H^1(\rn)+H_{W_p}^p(\rn)$.
Due to similarity,  we only consider the case $n(1/p-1)\in\nn$.
Consider first the case $f\in L^q_{\Phi_{p}(\cdot,1)}(\rn)
\cap H^{\Phi_{p}}(\rn)$. Without loss of generality,
we may also assume that $\|f\|_{H^{\Phi_p}(\rn)}=1$.

Let $q\in(1,\,\fz)$ and
$s\in\zz_+\cap [\lfloor n(1/p-1)\rfloor,\,\fz)$. Applying Lemma \ref{lem-CZd},  there exist
families $\{\Lambda_k\}_{k\in\zz}$ of index sets, $\{h^k_{i}\}_{k\in\zz,i\in\Lambda_k}$
of functions in $L^\fz(\rn)$ and $\{B_i^k\}_{k\in\zz,i\in\Lambda_k}$ of balls such that
\begin{align}\label{f-thm2-1}
f=\dsum_{k\in\zz}\dsum_{i\in\Lambda_k}h_i^k\qquad\textup{in}\;\;\cs'(\rn).
\end{align}
Define
\begin{align}\label{f-E}
\mathrm{E}:=\lf\{x\in\rn:\ f^*(x)<(1+|x|)^{-n}
[\log(e+|x|)]^{-p/(1-p)}\r\},
\end{align}
where $f^*$ denotes the non-tangential maximal
function as in \eqref{grand-funct} with $m:=n(1/p-1)$.
For any $k\in\zz$ and $i\in\Lambda_k$, define
$$B_{i,\mathrm{E}}^k:=B_i^k\cap \mathrm{E}\qquad\textup{and}\qquad
B_{i,\mathrm{E}^\complement}^k:=B_i^k\cap \mathrm{E}^\complement.$$
Let
$$
{\rm I}_0:=\lf\{(k,i):\ \lf|B_{i,\mathrm{E}}^k\r|\ge \frac12\lf|B_i^k\r|\r\}
\qquad\textup{and}\qquad
{\rm I}_1:=\lf\{(k,i):\ \lf|B_{i,\mathrm{E}^\complement}^k\r|\ge \frac12\lf|B_i^k\r|\r\}.
$$
It is easy to see that $\mathrm{I}_0\cap \mathrm{I}_1=\emptyset$
and
\begin{align*}
\dsum_{k\in\zz}\dsum_{i\in\Lambda_k}=\dsum_{(k,i)\in\mathrm{I}_0}
+\dsum_{(k,i)\in\mathrm{I}_1}.
\end{align*}

For any fixed $k_0\in\zz$, by Lemma \ref{lem-CZd},
it holds true that
\begin{align}\label{f-thm2-2}
\dsum_{(k_0,i)\in \mathrm{I}_0} |B_i^{k_0}|\le 2 \dsum_{(k_0,i)\in \mathrm{I}_0}\lf|B^{k_0}_{i,\mathrm{E}}\r|
\ls \lf|\lf\{x\in \mathrm{E}:\ f^*(x)>2^{k_0}\r\}\r|.
\end{align}
Similarly, for any $(k_0,i)\in \mathrm{I}_1$, using $W_p\in A_1(\rn)$ (see \eqref{f-A1}) and
$|B_{i,\mathrm{E}^\complement}^{k_0}|\ge \frac12|B_i^{k_0}|$, we obtain
\begin{align*}
\frac{W_{p}\lf(B_i^{k_0}\r)}{W_{p}\lf(B_{i,\mathrm{E}^\complement}^{k_0}\r)}
\ls \frac{\lf|B_i^{k_0}\r|}{\lf|B_{i,\mathrm{E}^\complement}^{k_0}\r|}\ls 1.
\end{align*}
This, combined with the same argument as in \eqref{f-thm2-2}, implies that
\begin{align}\label{f-thm2-3}
\dsum_{(k_0,i)\in \mathrm{I}_1} W_{p}\lf(B_i^{k_0}\r)\ls
\dsum_{(k_0,i)\in \mathrm{I}_1} W_{p}\lf(B_{i,\mathrm{E}^\complement}^{k_0}\r)
\ls  W_{p}\lf(\lf\{x\in \mathrm{E}^\complement:\ f^*(x)>2^{k_0}\r\}\r).
\end{align}

We split the decomposition in
\eqref{f-thm2-1} into
\begin{align*}
f&=\dsum_{(k,i)\in \mathrm{I}_0}h^k_i+\dsum_{(k,i)\in \mathrm{I}_1}h^k_i
=:\dsum_{(k,i)\in \mathrm{I}_0}\lz_{k,i}^{(0)}a_{k,i}^{(0)}+\dsum_{(k,i)\in \mathrm{I}_1}
\lz_{k,i}^{(1)}a_{k,i}^{(1)}=:f_0+f_1,
\end{align*}
where
$$
\begin{cases}
\lz_{k,i}^{(0)}:=c2^k|B_i^k|;\\
\lz_{k,i}^{(1)}:=c2^k[W_{p}(B_i^k)]^{1/p};\\
a_{k,i}^{(0)}:=h_i^k/\lz_{k,i}^{(0)};\\
a_{k,i}^{(1)}:=h_i^k/\lz_{k,i}^{(1)}
\end{cases}
$$
and
$c$ is the same as in (b) of Lemma \ref{lem-CZd}(iv) and it is independent of $k,\,i$ and $f$.
By Remark \ref{rem-at} and Lemma \ref{lem-CZd},
it is easy to see that
$a_{k,i}^{(0)}$ is a $(1,\,\fz,\,s)$-atom associated with the ball $B_i^k$ and
$a_{k,i}^{(1)}$ is a $(p,\,\fz,\,s)_{W_{p}}$-atom associated with $B_i^k$.
From \eqref{f-E} and \eqref{f-thm2-2}, it follows that
\begin{align*}
\dsum_{(k,i)\in \mathrm{I}_0}\lf|\lz_{k,i}^{(0)}\r|&\ls\sum_{k\in\zz} 2^k\lf|\lf\{x\in \mathrm{E}:\
f^*(x)>2^k\r\}\r|\\
&\ls \dint_{\mathrm{E}}f^*(x)\,dx\notag\\
&
\sim \dint_{\mathrm{E}}\frac{f^*(x)}{1+[f^*(x)(1+|x|)^n]^{1-p}
[\log(e+|x|)]^p}\,dx\notag\\
&\ls \dint_{\rn}\Phi_{p}\lf(x,\,f^*(x)\r)\,dx\ls 1,\notag
\end{align*}
where the last inequality follows from the assumption
$\|f\|_{H^{\Phi_p}(\rn)}=1$.
Further, using the atomic characterization of $H^1(\rn)$,
we know that $f_0\in H^1(\rn)$ and
$\|f_0\|_{H^1(\rn)}\ls 1 \sim\|f\|_{H^{\Phi_p}(\rn)}$.

For $f_1$, using \eqref{f-thm2-3}, we find that
\begin{align*}
\dsum_{(k,i)\in \mathrm{I}_1}\lf|\lz_{k,i}^{(1)}\r|^p
&\ls\sum_{k\in\zz} 2^{kp}W_{p}\lf(\lf\{x\in \mathrm{E}^\complement:\
f^*(x)>2^k\r\}\r)\\ \notag
&\ls \dint_{\mathrm{E}^\complement}\lf[f^*(x)\r]^p W_{p}(x)\,dx\\ \notag
&
\sim \dint_{\mathrm{E}^\complement}\frac{f^*(x)}{1+[f^*(x)(1+|x|)^n]^{1-p}
[\log(e+|x|)]^p}\,dx\notag\\
&\ls \dint_{\rn}\Phi_{p}\lf(x,\,f^*(x)\r)\,dx\ls 1,\notag
\end{align*}
Then, using the atomic characterization of $H_{W_p}^{{p}}(\rn)$ in Remark \ref{rem-at}, we know that $f_1\in H_{W_{p}}^p(\rn)$ and
$\|f_1\|_{H_{W_{p}}^p(\rn)}\ls 1 \sim\|f\|_{H^{\Phi_p}(\rn)}$.

Summarizing the above estimates gives us that  $L^q_{\Phi_{p}(\cdot,1)}(\rn)
\cap H^{\Phi_{p}}(\rn)\subset H^1(\rn)+H_{W_p}^p(\rn)$ and, for any $f\in L^q_{\Phi_{p}(\cdot,1)}(\rn)
\cap H^{\Phi_{p}}(\rn)$, there exist $f_0\in H^1(\rn)$ and $f_1\in H_{W_p}^p(\rn)$
such that $f=f_0+f_1$ in $\mathcal{S}'(\rn)$ and
\begin{align}\label{f-thm1-13}
\lf\|f_0\r\|_{H^1(\rn)}+\lf\|f_1\r\|_{H_{W_p}^p(\rn)}\ls \lf\|f\r\|_{H^{\Phi_p}(\rn)}.
\end{align}

For a general $f\in H^{\Phi_p}(\rn)$,
by Lemma \ref{lem-dp},
there exist $\{f_l\}_{l\in\nn}\subset L^q_{\Phi_1(\cdot,\,1)}(\rn)\cap H^{\Phi_p}(\rn)$ such that
$f=\sum_{l\in\nn} f_l$ in $H^{\Phi_p}(\rn)$ and
\begin{align*}
\|f_l\|_{H^{\Phi_p}(\rn)}\le 2^{2-l}\|f\|_{H^{\Phi_p}(\rn)}
\end{align*}
(see also \cite[p.\,138]{Ky14} for this fact).
Applying the previous argument to each $f_l$ with $l\in\nn$, we find
$f_{l,0}\in H^{1}(\rn)$ and $f_{l,1}\in H_{W_p}^p(\rn)$ such that
$f_l=f_{l,0}+f_{l,1}$ in $\mathcal{S}'(\rn)$, and
$$
\lf\|f_{l,0}\r\|_{H^1(\rn)}+\lf\|f_{l,1}\r\|_{H_{W_p}^p(\rn)}\ls \lf\|f_l\r\|_{H^{\Phi_p}(\rn)}\ls 2^{-l}\|f\|_{H^{\Phi_p}(\rn)}.
$$
Define $f_{0}:=\sum_{l\in\nn}f_{l,0}$
and $f_{1}:=\sum_{l\in\nn}f_{l,1}$.
It follows that $f_0\in H^1(\rn)$ and $f_{1}\in H_{W_p}^p(\rn)$, with
$$
\lf\|f_0\r\|_{H^1(\rn)}\le \sum_{l\in\nn} \lf\|f_{l,0}\r\|_{H^1(\rn)}
\ls \lf\|f\r\|_{H^{\Phi_p}(\rn)}
$$
and
$$
\lf\|f_1\r\|_{H_{W_p}^p(\rn)}^p\le\sum_{l\in\nn}\lf\|f_{l,1}\r\|_{H_{W_p}^p(\rn)}^p
\ls \sum_{l\in\nn} 2^{-lp}\|f\|_{H^{\Phi_p}(\rn)}^p\ls\|f\|_{H^{\Phi_p}(\rn)}^p.
$$
Altogether, we obtain $f=f_0+f_1$ in $\mathcal{S}'(\rn)$ and \eqref{f-thm1-13}.
This proves the inclusion $H^{\Phi_p}(\rn)\subset H^{1}(\rn)+H_{W_p}^p(\rn)$
in the case  $n(1/p-1)\in\nn$.

The proof of  $H^{\Phi_p}(\rn)\subset H^{1}(\rn)+H_{W_p}^p(\rn)$ in the case $n(1/p-1)\notin \nn$ is similar to the previous proof for the case $n(1/p-1)\in\nn$,
but now instead of \eqref{f-E} we define the set $\rm E$ as follows:
$$
\mathrm{E}:=\lf\{x\in\rn:\ f^*(x)<(1+|x|)^{-n}\r\}.
$$
The details are omitted.
This finishes the proof of Theorem \ref{thm1} when
$p\in(0,\,1)$.
\end{proof}

\begin{remark}\label{rem-ea}
Let $p\in(0,\,1)$ and $B$ be a ball in $\rn$.
For any $s\in \zz_+$ and $s\ge \lfloor n(1/p-1)\rfloor$, we know from Proposition \ref{prop-new}(iv) that a function $a$ on
$\rn$ is a $(\Phi_p , \fz, s)$-atom associated with $B$ if and only
if a is a $(p, \fz, s)_{W_p}$-atom associated with the same ball $B$.
On the other hand,  Theorem \ref{thm1} implies that $H_{W_p}^p\subsetneqq
H^{\Phi_p}(\rn)$. However, for any $p\in(0,\,1)$, from Lemma \ref{lem-whn} and the definitions of the dual spaces
of $H^{\Phi_p}(\rn)$ and $H_{W_p}^p(\rn)$ (see \cite[Theorem~3.2]{Ky14} for $p\in(n/(n+1), 1)$ and \cite[Theorem~3.5]{LY13} for general $p\in(0,1)$),
we deduce that
\begin{align*}
\lf[H^{\Phi_p}(\rn)\r]^*=\lf[H_{W_p}^p(\rn)\r]^*.
\end{align*}
This shows that the difference between $H^{\Phi_p}(\rn)$
and $H_{W_p}^p(\rn)$ is very small.
\end{remark}

\section{Proof of Theorem \ref{thm2}}\label{s4}

\hskip\parindent  Applying Theorem \ref{thm1}, we  prove Theorem \ref{thm2} in this section.

\begin{proof}[Proof of Theorem \ref{thm2}]
We first prove (i). Let $p\in(0,\,1]$ and $f\in H^{\Phi_p}(\rn)$. By Theorem \ref{thm1},
we know that there exist $f_0\in H^{\phi_0}(\rn)$ and $f_1\in H_{W_p}^p(\rn)$
such that $f=f_0+f_1$ in $\mathcal{S}'(\rn)$ and
$$
\lf\|f\r\|_{H^{\Phi_p}(\rn)}\sim \|f_0\|_{H^{\phi_0}(\rn)}+\|f_1\|_{H_{W_p}^p(\rn)}.
$$
Since $T$ is quasilinear and bounded on $H^{\phi_0}(\rn)$ and $H_{W_p}^p(\rn)$,
it follows, from \eqref{Phip}, that
\begin{align*}
\|T(f)\|_{H^{\Phi_p}(\rn)}&\ls \|T(f_0)\|_{H^{\Phi_p}(\rn)}+\|T(f_1)\|_{H^{\Phi_p}(\rn)}
\ls \|T(f_0)\|_{H^{{\phi_0}}(\rn)}+\|T(f_1)\|_{H_{W_p}^{p}(\rn)}\\
&\ls \|f_0\|_{H^{{\phi_0}}(\rn)}+\|f_1\|_{H_{W_p}^{p}(\rn)}\sim \|f\|_{H^{\Phi_p}(\rn)},
\end{align*}
which implies that $T$ is bounded on $H^{\Phi_p}(\rn)$. This shows (i).

The proof of (ii) is similar to that of (i) by applying the equivalence
established in Theorem \ref{thm1}(ii)
\begin{align*}
H^{\Phi_p}(\rn)=H^1(\rn)+H_{W_p}^p(\rn),
\end{align*}
the details being omitted. This finishes the proof of
Theorem \ref{thm2}.
\end{proof}

\bigskip

{\small\noindent Jun Cao

\noindent
Department of Applied Mathematics, Zhejiang University of Technology,
Hangzhou 310023, People's Republic of China

\noindent{\it E-mail:} \texttt{caojun1860@zjut.edu.cn}

\bigskip

\noindent {Liguang Liu (Corresponding author)}

\noindent  Department of Mathematics, School of Information, Renmin University
of China, Beijing 100872, China

\noindent {\it E-mail}: \texttt{liuliguang@ruc.edu.cn}

\bigskip

\noindent Dachun Yang and Wen Yuan

\noindent School of Mathematical Sciences, Beijing Normal
University, Laboratory of Mathematics and Complex Systems, Ministry
of Education, Beijing 100875, People's Republic of China

\noindent{\it E-mails:}
\texttt{dcyang@bnu.edu.cn} (D. Yang)

\hspace{0.85cm}\texttt{wenyuan@bnu.edu.cn} (W. Yuan)
}

\end{document}